\documentclass[a4paper,11pt]{article}
\usepackage[textwidth=14.5cm]{geometry}
\usepackage{graphicx}
\usepackage{amsmath,amssymb,amsthm}
\usepackage{hyperref}
\usepackage{authblk}
\newcommand{\pk}[1]{\mathrm{P}_{#1}}
\newcommand{\pro}[1]{\widehat{#1}_h}
\newcommand{\mybar}{\overline}
\newcommand{\ch}{\mathcal C_h}
\newcommand{\dzero}{\delta_0}
\newcommand{\eps}{\varepsilon}
\newcommand{\aintone}{\alpha_{*}}
\newcommand{\mone}{$\blacktriangle$}
\newcommand{\mtwo}{\Large$\bullet$\normalsize}
\newcommand{\mthree}{$\blacksquare$}
\newcommand{\mfour}{$\vartriangle$}
\newcommand{\mfive}{\Large$\circ$\normalsize}
\newcommand{\msix}{$\square$}
\theoremstyle{plain}
\newtheorem*{theoremNonum}{Theorem}
\newtheorem*{propositionNonum}{Proposition}
\newtheorem{lemma}{Lemma}
\theoremstyle{remark}
\newtheorem{remark}{Remark}
\theoremstyle{definition}
\newtheorem{hypo}{Hypothesis}
\newtheorem{exam}{Example}
\newtheorem*{scheme1}{Scheme OsTH}
\newtheorem*{scheme2}{Scheme OsPstab}
\newtheorem*{scheme3}{Scheme NSTH}
\newtheorem*{scheme4}{Scheme NSPstab}
\newcommand{\schemeone}{OsTH}
\newcommand{\schemetwo}{OsPstab}
\newcommand{\schemethree}{NSTH}
\newcommand{\schemefour}{NSPstab}
\begin{document}

\title{A viscosity-independent error estimate \\ of a pressure-stabilized Lagrange--Galerkin scheme \\ for the Oseen problem}
\author{Shinya Uchiumi} 
\affil{
	Department of Applied Mathematics,
	Waseda University, Tokyo, Japan\\ 
	\texttt{uchiumi@aoni.waseda.jp}
}

\date{\today}

\maketitle

\begin{abstract}
We consider a pressure-stabilized Lagrange--Galerkin scheme for the transient Oseen problem with small viscosity. 
In the scheme we use the equal-order approximation of order $k$ for both the velocity and pressure, and add a symmetric pressure stabilization term. 
We show an error estimate for the velocity with a constant independent of the viscosity if the exact solution is sufficiently smooth.
Numerical examples show high accuracy of the scheme for problems with small viscosity.

\end{abstract}

\section{Introduction} \label{sec:intro}
We consider a finite element scheme for the transient Oseen problem, known as linearizion of the Navier--Stokes problem, with small viscosity.
In this paper we construct a pressure-stabilized Lagrange--Galerkin (LG) scheme 
with higher-order elements, and show an error estimate independent of the viscosity.

When the viscosity is small, 
the finite element method suffers from two kinds of instabilities.
We begin with the issue of the material derivative.
In such case
the convection is dominated and it is important to put weight on information in the upwind direction to make schemes stable.
We here focus on the LG method, 
e.g. \cite{NotsuTabataOseen2015,NotsuTabata2016M2AN,Pironneau1982,Suli,TabataUchiumiNSA}, 
which is a combination of the characteristics method and the finite element method.
One of the advantages of it is that the resultant matrix is symmetric, which allows us to use efficient linear solvers \cite{benzi_golub_liesen_2005}.
Recently a LG method with a locally linearized velocity \cite{TSTEng} has been developed \cite{TabataUchiumiNSA} and convergence has been shown. The locally linearized velocity overcomes the difficulty in computing composite function term that appears in LG schemes.
In \cite{TabataUchiumiNSA} inf-sup stable elements \cite{2013_Boffi-Brezzi-Fortin} were used.

Besides the inf-sup stable elements, $\pk1/\pk1$-element  with a pressure stabilization term has been also used in LG methods,
where $\pk k/\pk l$ shows that we use the conforming triangular or tetrahedral element of order $k$ for the velocity and order $l$ for the pressure.
Notsu and Tabata have been proposed a LG scheme using the stabilization term of Brezzi and Pitk\"aranta \cite{BrezziPitkaranta} for the Navier--Stokes problem \cite{Notsu2008,NotsuTabata2008}, and analyzed the scheme for the Oseen problem and Navier--Stokes problem  \cite{NotsuTabataOseen2015,NotsuTabata2016M2AN}. 
Jia et al.~\cite{Jia2010} have been proposed and analyzed a LG scheme using the stabilization term of Bochev et al.~\cite{BochevEtal2006}.

Here we extend the $\pk1 / \pk1$ pressure-stabilized LG scheme to higher-order elements.
Simple symmetric stabilization terms for higher-order elements have been presented and applied to stationary problems in, e.g., \cite{BeckerBraack2001,BrezziFortin2001,Burman2008,DohrmannBochev2004,Silvester1994} and to the transient Stokes problem in \cite{BurmanFernandez2008}.
On the other hand, classical stabilization terms based on the residual of the momentum equations also have been studied for stationary problems in, e.g., \cite{DouglasWang1989,FrancaStenberg1991,HughesFrancaBalestra} and for the transient Stokes problem in \cite{JohnNovo2015}.
These terms are, however, rather complicated to implement compared to the symmetric stabilization especially for transient problems. 

Apart from the issue of the material derivative in the Oseen or Navier--Stokes problems,
dependence on the viscosity appears even in the Stokes problems.
Numerical solutions of the velocities contain approximation errors of the pressures multiplied by the inverse of the viscosity in standard finite element methods (e.g. \cite{SiamRDivergence}).
The grad-div stabilization \cite{GlowinskiLeTallec} is a choice to improve stability. 
Error analyses independent of
the viscosity were performed in \cite{OlshanskiiReusken} for the Stokes problem and in \cite{deFrutos2016} for the transient Oseen problem relying on this term.
In \cite{BermejoSaavedra2016} a LG scheme was developed for the Navier--Stokes problem with local projection stabilization that includes the grad-div term.

In this paper we use $\pk k / \pk k$-element, $k \geq 1$, and pressure-stabilization in the LG scheme for the transient Oseen problem, and show an error estimate independent of the viscosity.
In the scheme the symmetric pressure stabilization of Burman \cite{Burman2008} is used and symmetry of the LG method is inherited.
Although a pressure stabilized scheme for the transient Stokes problem has been analyzed by Burman and Fern\'andez \cite{BurmanFernandez2008},
we here pay attention to the constant of the stabilization term.
We consider the case where the viscosity $\nu$ is small and the exact solution is sufficiently smooth.
The error bound presented here is of order $\Delta t+h^2+h^k$ in the $L^2$-norm for the velocity and for $\nu^{1/2}$ times the gradient of the velocity, with constants independent of $\nu$.
Here, $\Delta t$ is a time increment, $h$ is a spacial mesh size.
This scheme is essentially unconditionally stable, that is, we can take $\Delta t$ and $h$ independently.
The grad-div stabilization is not needed in our analysis.
The technique used in our estimate is a projection of the exact solution of the velocity with the error independent of $\nu$.
A similar way was used by de Frutos et al.~\cite{deFrutos2016}.

This paper is organized as follows.
In the next section, after preparing notation, we state the Oseen problem and a pressure-stabilized LG scheme.
In Section \ref{sec:errorEstimate} we show an error estimate with a constant independent of the viscosity and give a proof.
In Section \ref{sec:numericalResults} we give some numerical results 
that show high accuracy for small viscosity and large pressures, 
and additionally show results of the Navier--Stokes problem.
In Section \ref{sec:conclusions} we give conclusions.
In Appendix we recall some lemmas used in the LG methods.

\section{A pressure-stabilized LG scheme for the Oseen problem}
\label{sec:continuousProb}

We prepare notation used throughout this paper, state the Oseen problem 
and then introduce our scheme.

Let $\Omega$ be a polygonal or polyhedral domain of $\mathbb R^d ~ (d=2,3)$.
We use the Sobolev spaces $W^{m,p}(\Omega)$ equipped with the norm $\|\cdot\|_{m,p}$ and the semi-norm $|\cdot|_{m,p}$ for $p\in [1,\infty]$ and a non-negative integer $m$.  
We denote $W^{0,p}(\Omega)$ by $L^p(\Omega)$.  
$W^{1,p}_0(\Omega)$ is the subspace of $W^{1,p}(\Omega)$ consisting of functions whose traces vanish on the boundary of $\Omega$.
When $p=2$, we denote $W^{m,2}(\Omega)$ by $H^m(\Omega)$ and drop the subscript $2$ in the corresponding norm and semi-norm.
For the vector-valued function $w\in W^{1,\infty}(\Omega)^d$ we define the semi-norm $|w|_{1,\infty}$ by
\begin{equation*}
\biggl\lVert\biggl[\sum_{i,j=1}^d \left( \frac{\partial w_i}{\partial x_j} \right)^2 \biggr]^{1/2}\biggr\rVert_{0,\infty}.
\end{equation*}
The pair of parentheses $(\cdot , \cdot)$ shows the $L^2(\Omega)^i$-inner product for $i=1, d$ or $d\times d$.
$L^2_0(\Omega)$ is the space of functions $q \in L^2(\Omega)$ satisfying $(q,1)=0$. 
We also use the notation $|\cdot|_{m,K}$ and $(\cdot,\cdot)_K$ for the semi-norm and the inner product on a set $K$.

Let $T>0$ be a time.
For a Sobolev space $X(\Omega)^i$, $i=1, d$,  
we use the abbreviations 
$H^m(X)=H^m(0,T;X(\Omega)^i)$ and $C(X)=C([0,T];X(\Omega)^i)$.
We define the function space $Z^m$ by
\begin{equation*}
\begin{split}
	Z^m &:= \{ v\in H^j(0,T; H^{m-j}(\Omega)^d); ~ j=0,\dots,m, \| v \|_{Z^m}<\infty \},\\
	\| v \|_{Z^m} &:= 
	\biggl( \sum_{j=0}^m \|v\|_{H^j(0,T;H^{m-j}(\Omega)^d)}^2 \biggr) ^{1/2}.
\end{split}
\end{equation*}
We also use the notation $H^m(t_1, t_2; X)$ and $Z^m(t_1,t_2)$ for spaces on a time interval $(t_1, t_2)$.

We consider the Oseen problem: find
$(u,p):\Omega \times (0,T) \to \mathbb R^d \times \mathbb R$ such that
\begin{equation}\label{Os} 
\begin{split}
\frac{\partial u}{\partial t}+(w \cdot \nabla) u 
- \nu \Delta u + \nabla p = f & \quad \text{in} \quad \Omega \times (0,T),\\
\nabla \cdot u = 0 & \quad \text{in} \quad \Omega \times (0,T),\\
u = 0 & \quad \text{on} \quad \partial \Omega \times (0,T),\\
u(\cdot,0) = u^0 & \quad \text{in} \quad \Omega,
\end{split}
\end{equation}
where 
$\partial \Omega$ represents the boundary of $\Omega$,
the constant $0<\nu\leq 1$ represents a viscosity, and
$w, f:\Omega \times (0,T) \to \mathbb R^d$ and
$u^0:\Omega \to \mathbb R^d$ are given functions.

We define the bilinear forms $a$ on $H^1_0(\Omega)^d\times H^1_0(\Omega)^d$ and $b$ on $H^1_0(\Omega)^d\times L^2_0(\Omega)$ by \index{$a$}\index{$b$}
\begin{equation*}
	a(u,v) := \nu(\nabla u,\nabla v), \quad
	b(v,q) := -(\nabla \cdot v,q).
\end{equation*}
Then, we can write the weak form of \eqref{Os} as follows:  
find $(u,p) : (0,T)\to H^1_0(\Omega)^d \times L^2_0(\Omega)$ such that for $t\in (0,T)$, 
\begin{subequations}\label{eq:Osweak}
\begin{align}
	\left( \Bigl( \frac{\partial u}{\partial t} + (w \cdot \nabla) u \Bigr) (t),v\right) + a(u(t),v)+b(v,p(t)) &= (f(t),v), \label{eq:Osweaka} \\ 
	&\forall v\in H^1_0(\Omega)^d,  \notag \\
	b(u(t),q) =0, \quad & \forall q\in L^2_0(\Omega),
\end{align}
\end{subequations}
with $u(0)=u^0$.

We introduce time discretization.
Let 
$\Delta t>0$ be a time increment, 
$N_T := \lfloor T/\Delta t \rfloor$ the number of time steps, 
$t^n := n\Delta t$, and 
$\psi^n := \psi(\cdot,t^n)$ for a function $\psi$ defined in $\Omega \times (0,T)$.
For a set of functions $\psi=\{ \psi^n \}_{n=0}^{N_T}$
we use two norms 
$\|\cdot\|_{\ell^\infty(L^2)}$ and $\|\cdot\|_{\ell^2(L^2)}$
defined by
\begin{equation*}
\begin{split}
\|\psi\|_{\ell^\infty(L^2)} & := 
\max \left\{ \|\psi^n\|_{  0  };n=0,\dots ,N_T \right\}, \\
\|\psi\|_{\ell^2(L^2)} & := 
\biggl(
\Delta t \sum_{n=1}^{N_T} \|\psi^n\|_{0}^2 
\biggr)^{1/2}.
\end{split}
\end{equation*}

Let $w$ be smooth. 
The characteristic curve $X(t; x,s)$ 
is defined by the solution of the system of the ordinary differential equations, 
\begin{equation} \label{eq:xode}
\begin{split}
	\frac{dX}{dt}(t; x,s)&=w(X(t;x,s),t), \quad t<s,\\
	X(s;x,s)&=x.
\end{split}
\end{equation}
Then, we can write the material derivative term $\frac{\partial u}{\partial t}+ (w\cdot \nabla) u$ as follows:
\begin{equation*}
\left(\frac{\partial u}{\partial t} + (w \cdot \nabla) u \right)(X(t),t)
=\frac{d}{dt} u(X(t),t).
\end{equation*}
For $w^*:\Omega \to \mathbb R^d$ we define the mapping $X_1(w^*):\Omega \to \mathbb R^d$ by
\begin{equation}\label{eq:x1def}
(X_1(w^*))(x) := x - w^*(x)\Delta t.
\end{equation}
\begin{remark}
	The image of $x$ by $X_1(w(\cdot,t))$ is nothing but the approximate value 
	of $X(t-\Delta t;x,t)$ obtained by solving \eqref{eq:xode} by the backward Euler method.
\end{remark}
Then, it holds that
\begin{equation*}
\frac{\partial u^n}{\partial t}+(w^n\cdot \nabla) u^n = \frac{u^n-u^{n-1}\circ X_1(w^{n-1})}{\Delta t} + O(\Delta t),
\end{equation*}
where the symbol $\circ$ stands for the composition of functions, 
e.g., $(g\circ f)(x) := g(f(x))$.

We next introduce spacial discretization.
Let $\{\mathcal T_h\}_{h \downarrow 0}$ be a regular family of triangulations of $\mybar \Omega$ \cite{Ciarlet}, $h_K:=\operatorname{diam}(K)$ for an element $K\in \mathcal T_h$, and $h:= \max_{K\in \mathcal T_h} h_K$. 
For a positive integer $m$, the finite element space of order $m$ is defined by
\[
	W_h^{(m)} := \{ \psi_h \in C(\mybar\Omega); ~ \psi_{h|K} \in \pk m(K), ~ \forall K \in \mathcal T_h \},
\]
where $\pk m(K)$ is the set of polynomials on $K$ whose degrees are equal to or less than $m$.
Let $\Pi_h^{(m)}: C(\mybar \Omega) \to W_h^{(m)}$ be the Lagrange interpolation operator, which is naturally extended to vector-valued functions.

We begin with a scheme using the standard $\pk k/\pk{k-1}$-finite element, which is called (generalized) Taylor--Hood element.
Let
\begin{equation} \label{eq:elemTH}
	V_h \times \mybar Q_h := ((W_h^{(k)})^d \cap H_0^1(\Omega)^d) \times (W_h^{(k-1)} \cap L_0^2(\Omega))
\end{equation} 
be the $\pk k/\pk{k-1}$-finite element space for $k \geq 2$.
The LG scheme with a locally linearized velocity and this Taylor--Hood element for the Oseen problem (\schemeone) is stated as follows:

\begin{scheme1} 
	Let $u_h^0 \in V_h$ be an approximation of $u^0$.
	Find $\{(u_h^n, p_h^n)\}_{n=1}^{N_T} \subset V_h \times \mybar Q_h$ such that
	\begin{subequations} \label{eq:scheme1}
		\begin{alignat}{2}
		\biggl( \frac{u_h^{n} - u_h^{n-1}\circ X_{1}(\Pi_h^{(1)} w^{n-1})}{\Delta t}, v_h \biggr) + a(u_h^n, v_h) + & b(v_h, p_h^n) \qquad \notag \\ & = (f^n,v_h),  ~ && \forall v_h \in V_h, 
		\\
		b(u_h^n, q_h) &= 0, ~ && \forall q_h \in \mybar Q_h, 
		\end{alignat}
	\end{subequations}
\end{scheme1}

When $k=2$, this type of scheme for the Navier--Stokes problem has already been introduced and analyzed in \cite{TabataUchiumiNSA}. 
In the mapping $X_1(\cdot)$, a locally linearized velocity $\Pi_h^{(1)} w^{n-1}$ is used instead of the original velocity $w^{n-1}$.
If the original velocity is used, it is difficult to evaluate the exact value of integration.
The next proposition assures that the scheme with the locally linearized velocity is exactly computable.

\begin{propositionNonum}[\cite{TabataUchiumi1,TabataUchiumiNSA}]
\label{prop:exactly}
Let $u_h$, $v_h \in (W_h^{(m)})^d$ for a positive integer $m$.
Suppose that 
\begin{equation} \label{eq:assumptionExact}
	w^* \in W^{1,\infty}_0(\Omega)^d \quad \text{and} \quad \aintone \Delta t |w^*|_{1,\infty} < 1, 
\end{equation}
where $\aintone$ is the constant defined in \eqref{eq:aintone} below.
Then, $\int_{\Omega} (u_h \circ X_{1}(\Pi_h^{(1)} w^{*})) \cdot v_h dx$ is exactly computable.
\end{propositionNonum}

With the assumption \eqref{eq:assumptionExact} for $w^* = w^{n-1}$ at each step $n$, \eqref{eq:scheme1} is exactly computable thanks to 
Proposition.
In \cite{TabataUchiumiNSA} the authors have analyzed the scheme to show 
the estimates
\[
\| \nabla(u_h -u) \|_{\ell^\infty(L^2)}, ~ \| p_h - p \|_{\ell^2(L^2)} 
\leq c (\nu^{-1}) (\Delta t + h^2),
\]
where the constant $c$ depends on $\nu^{-1}$ exponentially.

Here we use the equal-order element with pressure stabilization.
Let 
\[ V_h \times Q_h := ((W_h^{(k)})^d \cap H_0^1(\Omega)^d) \times (W_h^{(k)} \cap L_0^2(\Omega))\] 
be the equal-order $\pk k/\pk{k}$-finite element space for $k \geq 1$.
We define a pressure stabilization term $\ch: Q_h \times Q_h \to \mathbb R$, which enables us to use the equal-order element, by
\begin{equation*}
\begin{split}
	\ch(p_h, q_h) := \sum_{K\in \mathcal T_h} h_K^{2k} \sum_{|\alpha|=k} (D^\alpha p_h, D^\alpha q_h)_K,
\end{split}
\end{equation*}
where $\alpha$ is the multi-index and $D^\alpha$ is the partial differential operator.
We define the corresponding semi-norm on $Q_h$ by
\begin{equation} \label{eq:semi-norm} 
	|q_h|_h := \ch(q_h,q_h)^{1/2} 
	= \left( \sum_{K \in \mathcal T_h} h_K^{2k} |q_h|_{k,K}^2 \right)^{1/2}.  
\end{equation}

\begin{remark}
The term $\ch$ introduced by Burman \cite{Burman2008} is an extension of that by Brezzi and Pitk\"aranta \cite{BrezziPitkaranta} for the $\pk1/\pk1$-element to higher order elements.
For the stabilization term, instead of $\ch$, we can also choose another positive semi-definite bilinear form whose corresponding semi-norm is equivalent to \eqref{eq:semi-norm}. 
Examples include the terms in \cite{BeckerBraack2001,DohrmannBochev2004,Silvester1994}, as pointed out in \cite{Burman2008}.
\end{remark}

We are now in position to state a pressure-stabilized LG scheme for the Oseen problem (\schemetwo).

\begin{scheme2} 
Let $u_h^0 \in V_h$ be an approximation of $u^0$.
Find $\{(u_h^n, p_h^n)\}_{n=1}^{N_T} \subset V_h \times Q_h$ such that
\begin{subequations} \label{eq:scheme}
\begin{align}
	\biggl( \frac{u_h^{n} - u_h^{n-1}\circ X_{1}(\Pi_h^{(1)} w^{n-1})}{\Delta t}, v_h \biggr) + a(u_h^n, v_h) + b(v_h, p_h^n) \qquad \notag \\ = (f^n,v_h),  ~ \forall v_h \in V_h, \label{eq:schemea}\\
	b(u_h^n, q_h) - \dzero \ch(p_h^n, q_h) = 0, ~ \forall q_h \in Q_h, \label{eq:schemeb}
\end{align}
\end{subequations}
where 
$\dzero>0$ is a stabilization parameter.
\end{scheme2}

With the assumption \eqref{eq:assumptionExact} for $w^* = w^{n-1}$ at each step $n$, \eqref{eq:scheme} is exactly  computable and has a unique solution $(u_h^n, p_h^n)$ thanks to 
Proposition
and the stabilization term $\ch$ \cite{Burman2008}.
The error introduce by the locally linearized velocity is properly estimated in Theorem below.

\begin{remark}
\begin{enumerate}
\item In Scheme 
\schemetwo,
the resultant matrix to be solved is symmetric, which enables us to use efficient linear solvers \cite{benzi_golub_liesen_2005}.
\item When $k=1$, Notsu and Tabata \cite{NotsuTabataOseen2015} proposed and analyzed a pressure-stabilized LG scheme, where the locally linearized velocity was not introduced. 
\item Burman and Fern\'andez \cite{BurmanFernandez2008} proposed and analyzed a scheme for the transient Stokes problem 
using a same type of pressure stabilization.
Since in their choice the stabilization parameter $\dzero$ is proportional to $1/\nu$,  
it seems to be difficult to get error estimates independent of $\nu$, which we will show in the next section.
\end{enumerate}
\end{remark}

\section{An error estimate focused on the viscosity for the Oseen problem}
\label{sec:errorEstimate}

Before stating the result we introduce hypotheses.

\begin{hypo}
	\label{hypo:reg}
The velocity $w$ and the exact solution $(u,p)$ of the Oseen problem \eqref{Os} satisfies
\begin{align*}
	w  \in C( W^{1,\infty}_0 \cap W^{2,\infty} ) \cap H^1(L^\infty), ~ 
	u  \in Z^2 \cap H^1(H^{k+1}), ~
	p \in C(H^{k+1}).
\end{align*}
\end{hypo}

\begin{hypo} \label{hypo:dt}
The time increment $\Delta t$ satisfies $0<\Delta t \leq \Delta t_0$, where
\[
	\Delta t_0 := \frac{1}{4 \aintone |w|_{C(W^{1,\infty})}},
\]
and $\aintone$ is the constant defined in \eqref{eq:aintone} below.
\end{hypo}

\begin{hypo}[Triangulation] \label{hypo:mesh}
Every element $K\in \mathcal T_h$ has at least one internal vertex. 
\end{hypo}

\begin{hypo}[Choice of the initial value] \label{hypo:initial}
There exists a positive constant $c$ independent of $h$ such that
\[
	\| u_h^0 -u^0 \|_0 \leq ch^k |u^0|_k.
\]
\end{hypo}

\begin{remark} 
\begin{enumerate}
\item Hypothesis \ref{hypo:reg} implies that $u \in C(H^{k+1})$ and $u^0 \in H^{k+1}(\Omega)^d$.
\item Under Hypotheses \ref{hypo:reg} and \ref{hypo:dt}, the property \eqref{eq:assumptionExact} for $w^*=w^{n-1}$ at each step $n$ is clearly satisfied.
\item Hypothesis \ref{hypo:initial} is satisfied if we take $u_h^0$ as the Lagrange 
interpolation of $u^0$, for example.
\end{enumerate}
\end{remark}

\begin{theoremNonum} \label{theo:main}
Let $V_h \times Q_h$ be the $\pk k / \pk k$-finite element space for $k \geq 1$.
Suppose Hypotheses \ref{hypo:reg}--\ref{hypo:initial}.
Let $(u_h, p_h) := \{(u_h^n, p_h^n)\}_{n=0}^{N_T}$ be the solution of Scheme \schemetwo.
Then there exists a positive constant $c_*$ independent of $\nu$, $h$, $\Delta t$ such that
\begin{equation} \label{eq:mainEstimate}
	\| u_h-u \|_{\ell^\infty(L^2)}, \sqrt{\nu} \| \nabla(u_h-u) \|_{\ell^2(L^2)}
	\leq c_*( \Delta t + h^2 + h^k). 
\end{equation}
\end{theoremNonum}

\begin{remark}
\begin{enumerate}
\item The constant $c_*$ depends on Sobolev norms of $w, u$ and $p$, and the stabilization parameter $\dzero$. 
Note also that we assumed that $\nu \leq 1$.
The parameter $\dzero$ should not depend on $\nu$ from the viewpoint of this estimate.
\item If $\pk k/\pk{k-1}$-element is employed, we have an estimate of the same order $\Delta t + h^2 + h^k$, but 
it seems to be difficult to remove the dependence on the viscosity,
which is observed in the numerical experiments in Section \ref{sec:numericalResults}.
\item The term $h^2$ appears in \eqref{eq:mainEstimate} because of the introduction of the locally linearized velocity.
\item 
It seems to be difficult to derive the estimate $h^{k+1}$ for the spacial discretization in $\ell^\infty(L^2)$ independent of the viscosity.
Although another type of Stokes projection yields an estimate of order $h^{k+1}$, e.g. \cite{NotsuTabataOseen2015}, the projection error contains the dependence.
de Frutos et al.~\cite{deFrutos2016} derived the same order $h^k$ as ours 
independent of the viscosity
for the backward Euler method or the BDF2 formula with the grad-div stabilization. 
\item Here we do not discuss a estimate for the pressure.
de Frutos et al.~\cite{deFrutos2016} used inf-sup stable elements and show an estimate for the pressure independent of the viscosity. 
However, further discussion seems to be necessary for the pressure-stabilized method.
\item When $k=1$, Notsu and Tabata \cite{NotsuTabataOseen2015} analyzed the pressure-stabilized LG scheme without the locally linearized velocity. 
They derived the estimates
\[
	\| \nabla(u_h -u) \|_{\ell^\infty(L^2)}, ~ \| p_h - p \|_{\ell^2(L^2)} 
	\leq c (\nu^{-1}) (\Delta t + h),
\]
where the constant $c$ depends on $\nu^{-1}$ exponentially.
\end{enumerate}
\end{remark}

Before the proof we prepare some lemmas.
First we recall a discrete version of the Gronwall inequality.

\begin{lemma}[discrete Gronwall inequality] 
\label{lem:discreteGronwall}
Let $a_0$ and $a_1$ be non-negative numbers, $\Delta t \in (0,\frac{1}{2a_0}]$ be a real number, and $\{x^n\}_{n\geq 0}, \{y^n\}_{n\geq 1}$ and $\{b^n\}_{n\geq 1}$ be non-negative sequences.
Suppose
\begin{equation*}
	\frac{x^n-x^{n-1}}{\Delta t} + y^n \leq a_0 x^n + a_1 x^{n-1} + b^n, ~ \forall n \geq 1.
\end{equation*}
Then, it holds that 
\begin{equation*}
	x^n + \Delta t \sum_{i=1}^n y^i \leq \exp [ (2a_0+a_1)n\Delta t ] \left( x^0 + \Delta t \sum_{i=1}^n b^i \right), 
	~ \forall n \geq 1.
\end{equation*}
\end{lemma}
Lemma \ref{lem:discreteGronwall} is shown by using the inequalities
\begin{equation*}
	\frac{1}{1-a_0\Delta t} \leq 1+2a_0\Delta t \leq \exp(2a_0\Delta t).
\end{equation*} 

In Lemmas \ref{lemm:interEst}--\ref{lemm:Best} below, the constants $c$ are independent of $h$.

We recall the fundamental properties of Lagrange and Cl\'ement interpolations \cite{Ciarlet,Clement}.

\begin{lemma} \label{lemm:interEst}
Suppose that $\{ \mathcal T_h \}_{h \downarrow 0}$ is a regular family of triangulations of $\mybar \Omega$. \\
(i) Let $\Pi_h^{(m)}: C(\mybar \Omega)^i \to (W_h^{(m)})^i$, $i=1,d$, be the Lagrange interpolation operator to $\pk m$-finite element space for a positive integer $m$.
Then 
there exist positive constants $\aintone \geq 1$ and $c$ independent of $h$ such that
\begin{align}
	|\Pi_h^{(1)} w|_{1,\infty} &\leq \aintone |w|_{1,\infty} ,         \quad \forall w \in W^{1,\infty}(\Omega)^d, \label{eq:aintone}\\
	\| \Pi_h^{(1)} w - w\|_{0,\infty} &\leq c h^{2} |w|_{2,\infty},
		 ~ \forall w \in W^{2,\infty}(\Omega)^d, \notag  \\
	\| \Pi_h^{(m)} w - w \|_{0,K} & \leq ch_K^{m+1} |w|_{m+1,K}, ~ \forall K \in \mathcal T_h, \forall w \in H^{m+1}(K)^i, i=1,d. \notag
\end{align}
(ii) Let $\Pi_{h,C}^{(m)}: L^2(\Omega) \to W_h^{(m)}$ be the Cl\'ement interpolation operator to $\pk m$-finite element space for a positive integer $m$. 
Then there exists a positive constants $c$ such that 
\begin{alignat*}{2}
	| \Pi_{h,C}^{(m)} \psi - \psi |_1 &\leq c h^m |\psi|_{m+1}, && ~ \forall \psi \in H^{m+1}(\Omega), \\
	\biggl( \sum_{K \in \mathcal T_h}| \Pi_{h,C}^{(m)} \psi|_{m, K}^2 \biggr)^{1/2} &\leq c |\psi|_{m}, && ~ \forall \psi \in H^{m}(\Omega).
\end{alignat*}
\end{lemma}

When $k\geq 2$, we use the auxiliary $\pk{k-1}$-pressure space $\mybar Q_h$ defined in \eqref{eq:elemTH},
and
$(\pro{z},\pro{r}) \in V_h \times \mybar Q_h$ 
be the Stokes projection of $(z,r) \in H^1_0(\Omega)^d \times L^2_0(\Omega)$ for the fixed viscosity $\nu=1$ defined by 
\begin{subequations}
	\label{eq:stokesproj}
	\begin{align}
	&& (\nabla \pro z, \nabla v_h) - (\nabla \cdot v_h, \pro r) &= (\nabla z, \nabla v_h) - (\nabla \cdot v_h, r), \quad &\forall v_h\in V_h, \\
	&& -(\nabla \cdot \pro z, \mybar q_h)&= -( \nabla \cdot z, \mybar q_h),  &\forall \mybar q_h \in \mybar Q_h. \label{eq:stokesprojb}
	\end{align}
\end{subequations}

\begin{lemma}
	\label{lemm:infsup}
Suppose that $\{ \mathcal T_h \}_{h \downarrow 0}$ is a regular family of triangulations of $\mybar \Omega$ and Hypothesis \ref{hypo:mesh}.
Let $V_h \times \mybar Q_h$ 
be the $\pk k/\pk{k-1}$-finite element space for $k \geq 2$.
Then, 
there exists a positive constant $c$ such that
\begin{equation}
	\|\pro{z}-z\|_{1}, \| \pro{r}-r \|_0 \leq 
	c h^k 
	(|z|_{k+1} + |r|_k),
\end{equation}
where $(\pro{z},\pro{r}) \in V_h \times \mybar Q_h$ is the Stokes projection of $(z,r) \in (H^1_0(\Omega)^d \cap H^{k+1}(\Omega)^d) \times (L^2_0(\Omega) \cap H^k(\Omega))$ defined in \eqref{eq:stokesproj}. 
\end{lemma}
This estimate is a direct consequence of the inf-sup stability for the $\pk k / \pk{k-1}$-element \cite{2013_Boffi-Brezzi-Fortin}. 
Since in \eqref{eq:stokesproj} the fixed viscosity is used,
we have the estimate of the projection independent of the viscosity.

\begin{lemma} \label{lemm:Best}
Suppose that $z \in H^{k+1}(\Omega)^d$ satisfies $\nabla \cdot z =0$,
$\{ \mathcal T_h \}_{h \downarrow 0}$ is a regular family of triangulations of $\mybar \Omega$
and Hypothesis \ref{hypo:mesh}.
Let $V_h \times Q_h$ be the $\pk k / \pk k $-finite element space for $k\geq 1$.
Let $z_h \in V_h$ be the Lagrange interpolation of $z$ when $k=1$, or
the first component of the Stokes projection of $(z,0)$ defined in \eqref{eq:stokesproj} when $k \geq 2$.
Then, there exists a positive constant $c$ such that
\begin{equation} \label{eq:lemmBest}
	b(z_h, q_h) \leq ch^k |z|_{k+1} |q_h|_h, ~ \forall q_h \in Q_h,
\end{equation}
where the semi-norm $|\cdot|_h$ is defined in \eqref{eq:semi-norm}
\end{lemma}

\begin{proof}
When $k=1$, by using $\nabla \cdot z=0$, the integration by part and Lemma \ref{lemm:interEst}, we get the estimate \eqref{eq:lemmBest} as follows:
\begin{equation*}
\begin{split}
	b(z_h,q_h) &= b(z_h-z,q_h) = (z_h-z, \nabla q_h) \leq c \sum_{K \in \mathcal T_h} h_K^2 |z|_{2,K} \| \nabla q_h \|_{0,K} \\
	&\leq ch |z|_{2} |q_h|_h.
\end{split}
\end{equation*}
When $k \geq 2$, it holds that from \eqref{eq:stokesprojb} and Lemma \ref{lemm:interEst}
\[
	b(z_h, \mybar q_h) = b(z, \mybar q_h) = 0, ~ \forall \mybar q_h \in \mybar Q_h, \text{ and }
\]
\[
	\|q_h - \mybar \Pi_h^{(k-1)} q_h  \|_0 \leq \|q_h - \Pi_h^{(k-1)} q_h  \|_0 \leq c \left( \sum_{K \in \mathcal T_h} h_K^{2k} |q_h|_{k,K}^2 \right)^{1/2},
\]
where $\mybar \Pi_h^{(k-1)} q_h \in \mybar Q_h$ is the Lagrange interpolation of $q_h$ with the correction of the constant so that $\mybar \Pi_h^{(k-1)} q_h \in L^2_0(\Omega)$.
By Lemma \ref{lemm:infsup}
we get the estimate \eqref{eq:lemmBest} as follows:
\begin{equation*}
\begin{split}
	& b(z_h, q_h) = b(z_h, q_h - \mybar \Pi_h^{(k-1)} q_h) 
	\\
	\leq & \| \nabla \cdot (z_h - z)\|_0 \|q_h - \mybar \Pi_h^{(k-1)} q_h  \|_0 
	\leq ch^k |z|_{k+1} |q_h|_h.
\end{split}
\end{equation*}
\end{proof}

We now begin the proof of Theorem,
where we also refer to Lemmas \ref{lemm:bijective_1}--\ref{lemm:estR3pre} in Appendix for properties of the mapping $X_1(\cdot)$. 

\begin{proof}[Proof of Theorem] 
Here we simply write  $X_{1h}^{n-1}=X_{1}(\Pi_h^{(1)} w^{n-1})$.
We use $c$ to represent a generic positive constant that 
is independent of $\nu$, $\Delta t$ and $h$ but depends on Sobolev norms $\|u\|_{C(H^1)}$ and $\|w\|_{C(W^{2,\infty})}$, and
may take a different value at each occurrence.
	
Let $z_h(t) \in V_h$ be, as in Lemma \ref{lemm:Best}, the Lagrange interpolation of $u(t)$ when $k=1$, or
the first component of the Stokes projection of $(u(t),0)$ defined in \eqref{eq:stokesproj} when $k \geq 2$, 
and let $r_h(t) \in Q_h$ be the Cl\'ement interpolation of $p(t)$ with the correction of the constant so that $r_h(t) \in L^2_0(\Omega)$.
We define the error terms by
\[
	(e_h^n, \eps_h^n):=(u_h^n - z_h^n, p_h^n - r_h^n), ~ \eta(t) := u(t) - z_h(t).
\]
From \eqref{eq:schemea}, \eqref{eq:Osweaka} with $t=t^n$ and $v=v_h$, 
and \eqref{eq:schemeb},
we have an error equations in $(e_h^n, \eps_h^n)$:

\begin{subequations}\label{eq:errorWhole}
\begin{align} \label{eq:errorMain}
	& \left(\frac{e_h^n - e_h^{n-1} \circ X_{1h}^{n-1}}{\Delta t}, v_h\right) + a(e_h^n,v_h)+b(v_h,\eps_h^n) \notag \\  
	& = (R^n, v_h) + a(\eta^n, v_h) + b(v_h, p^n - r_h^n), &\forall v_h \in V_h,\\
	& b(e_h^n, q_h) - \dzero \ch(\eps_h^n, q_h)  = -b(z_h^n, q_h) + \dzero \ch(r_h^n, q_h), &\forall q_h \in Q_h, \label{errorMainb}
	\end{align}
\end{subequations}
for $n=1,\dots,N_T$, where $R^n:=R^n_1+R^n_2+R^n_3$,
\begin{equation}\label{eq:defR}
\begin{split}
	R_{1}^n & := \frac{\partial u^n}{\partial t} + (w^n \cdot \nabla)u^n 
	-\frac{u^n-u^{n-1}\circ X_1(w^{n-1})}{\Delta t}, \\
	R_{2}^n & := \frac{u^{n-1}\circ X_{1h}^{n-1} - u^{n-1} \circ X_1(w^{n-1})}{\Delta t}, \\
	R_{3}^n & := \frac{\eta^n-\eta^{n-1}\circ X_{1h}^{n-1}}{\Delta t}.
\end{split}
\end{equation}
Substituting $(v_h, q_h) = (e_h^n, \eps_h^n)$ in \eqref{eq:errorWhole} 
and using the identity $(a-b)a = 1/2(a^2-b^2+(a-b)^2)$
yields
\begin{equation} \label{eq:erroreq2}
\begin{split}
	&\frac{1}{2\Delta t} \bigl( \| e_h^n \|_0^2 - \| e_h^{n-1}\circ X_{1h}^{n-1}) \|_0^2 + \| e_h^n - e_h^{n-1} \circ X_{1h}^{n-1} \|_0^2 \bigr) 
	+ \nu \| \nabla e_h^n \|_0^2 + \dzero |\eps_h^n|_h^2 \\
	&= (R^n,e_h^n) + a(\eta^n, e_h^n) + b(e_h^n, p^n - r_h^n) + b(z_h^n, \eps_h^n) - \dzero \ch(r_h^n, \eps_h^n).
\end{split}
\end{equation}

We now estimate the terms in \eqref{eq:erroreq2}.
With Hypothesis \ref{hypo:dt} and the properties
\begin{equation}
	\Pi_h^{(1)} w^{n-1} \in W^{1,\infty}_0(\Omega)^d \quad \text{and} \quad |\Pi_h^{(1)} w^{n-1}|_{1,\infty} \Delta t \leq \aintone |w^{n-1}|_{1,\infty} \Delta t \leq 1/4,
\end{equation}
we use Lemma \ref{lemm:bijective_1} in Appendix to have
\begin{equation} \label{eq:errorEqRT}
	\| e_h^{n-1}\circ X_{1h}^{n-1} \|_0^2 \leq (1+c\Delta t) \| e_h^{n-1} \|_0^2.
\end{equation}

After applying Schwarz's inequality to $(R^n, e_h^n)$, 
we estimate $\|R_i^n\|_0$, $i=1,2,3$.
By Lemma \ref{lemm:estR1} in Appendix,
\begin{equation} \label{eq:errorEqResest1}
	\|R_1^n\|_0 \leq c \sqrt{\Delta t} \biggl( \|u\|_{Z^2(t^{n-1},t^n)} + \left\| \frac{\partial w}{\partial t} \right\|_{L^2(t^{n-1}, t^n; L^\infty)} \biggr).
\end{equation} 
By Lemma \ref{lemm:twofunc} in Appendix with $q=2$, $p=\infty$, $p'=1$, $w_1=\Pi_h^{(1)} w^{n-1}$, $w_2=w^{n-1}$ and $v=u^{n-1}$, and by Lemma \ref{lemm:interEst},
\begin{equation} \label{eq:errorEqResest2}
	\|R_2^n\|_0 \leq c \| \Pi_h^{(1)} w^{n-1} - w^{n-1} \|_{0,\infty} 
	\leq c h^2.
\end{equation}
By Lemma \ref{lemm:estR3pre} in Appendix with $v=\eta$ and $w^*=\Pi_h^{(1)} w^{n-1}$, and by Lemma \ref{lemm:interEst} or \ref{lemm:infsup},
\begin{equation} \label{eq:errorEqResest3}
\begin{split}
	\|R_3^n\|_0 
	&\leq \frac{c}{\sqrt{\Delta t}} \biggl( \left\| \frac{\partial \eta}{\partial t} \right\|_{L^2(t^{n-1}, t^n; L^2)} + \|\nabla \eta\|_{L^2(t^{n-1}, t^n; L^2)} \biggr) \\
	& \leq \frac{c h^k}{\sqrt{\Delta t}}  \biggl( \| u \|_{H^1(t^{n-1}, t^n; H^{k+1})} 
	\biggr).
\end{split}
\end{equation}

An estimate for $a$ is easily obtained by Lemma \ref{lemm:interEst} or \ref{lemm:infsup}:
\begin{equation} \label{eq:errorEqaest}
	a(\eta^n, e_h^n) \leq \frac{\nu}{2} \|\nabla \eta^n\|_{0}^2 + \frac{\nu}{2} \|\nabla e_h^n\|_{0}^2 
	\leq ch^{2k} |u^n|_{k+1}^2 + \frac{\nu}{2} \|\nabla e_h^n\|_{0}^2,
\end{equation}
where we note that we assumed $\nu \leq 1$.
The integration by part and Lemma \ref{lemm:interEst}-(ii) yields
\begin{equation} \label{eq:errorEqBAppest}
\begin{split}
	b(e_h^n, p^n - r_h^n) &= (e_h^n, \nabla(p^n - r_h^n)) 
	\leq \frac{1}{2} \|e_h^n\|_0^2 + \frac{1}{2} \| \nabla (p^n - r_h^n) \|_0^2 \\
	&\leq \frac{1}{2} \|e_h^n\|_0^2 + ch^{2k} |p^n|_{k+1}^2.
\end{split}
\end{equation}
By Lemma \ref{lemm:Best},
\begin{equation} \label{eq:errorEqBStabest}
\begin{split}
	b(z_h^n, \eps_h^n) \leq  ch^k |u^n|_{k+1} |\eps_h^n|_h
	\leq \frac{c}{\dzero} h^{2k} |u^n|_{k+1}^2 + \frac{\dzero}{2}  |\eps_h^n|_h^2.
\end{split}
\end{equation}
By using stability of Cl\'ement interpolation (Lemma \ref{lemm:interEst}-(ii)),
\begin{equation} \label{eq:errorEqCest}
	- \dzero \ch(r_h^n, \eps_h^n) \leq \dzero |r_h^n|_h |\eps_h^n|_h 
	\leq \frac{\dzero}{2} |r_h^n|_h^2 + \frac{\dzero}{2} |\eps_h^n|_h^2 
	\leq \frac{c\dzero}{2} h^{2k} |p^n|_{k}^2 + \frac{\dzero}{2} |\eps_h^n|_h^2.
\end{equation}
Gathering the estimates \eqref{eq:errorEqRT}--\eqref{eq:errorEqCest}, from \eqref{eq:erroreq2} we get
\begin{align*}
	&\frac{1}{2\Delta t}(\|e_h^n\|_0^2 - \|e_h^{n-1}\|_0^2) + \frac{\nu}{2} \|\nabla e_h^n \|_0^2 
	\leq c \|e_h^{n-1}\|_0^2 + \eps_0 \|e_h^n\|_0^2 \\
	&+ c \biggl\{ \Delta t \Bigl( \|u\|_{Z^2(t^{n-1},t^n)}^2 + \left\| \frac{\partial w}{\partial t} \right\|_{L^2(t^{n-1}, t^n; L^\infty)}^2 \Bigr) 
	\\ & \qquad 
	+ \frac{h^{2k}}{\Delta t} \|  u \|_{H^1(t^{n-1}, t^n; H^{k+1})}^2 + h^4 
	+ h^{2k}[(1+\dzero^{-1})  \|u^n\|_{k+1}^2 + (1+\dzero)\|p^n\|_{k+1}^2 ] \biggr\},
\end{align*}
where the positive constant $\eps_0$ has been chosen so that $\Delta t_0 \leq \frac{1}{4\eps_0}$.
We now apply the discrete Gronwall's inequality (Lemma \ref{lem:discreteGronwall}) 
to get for $1\leq n \leq N_T$
\begin{align*}
	& \|e_h^n\|_0^2 + \nu \Delta t \sum_{j=1}^n \|\nabla e_h^j\|_0^2 \\
	\leq & c \exp\{ c' n\Delta t \}(\Delta t^2 + h^{2k} + h^4) 
	\biggl[ \left\| \frac{\partial w}{\partial t} \right\|_{L^2(0, t^n; L^\infty)}^2 + \|u\|_{Z^2(0,t^n)}^2 + \|  u \|_{H^1(0, t^n; H^{k+1})}^2 \\
	&+ (1+\dzero^{-1}) \Delta t \sum_{j=1}^n \|u^j\|_{k+1}^2 + (1+\dzero) \Delta t \sum_{j=1}^n \|p^j\|_{k+1}^2 
	+ \| u^0 \|_{k+1}^2
	\biggr],
\end{align*}
where we have used Hypothesis \ref{hypo:initial} for the initial value.
We have the conclusion by the triangle inequalities,
\begin{align*}
	\| u_h - u \|_{\ell^\infty(L^2)} &\leq \|e_h\|_{\ell^\infty(L^2)} + \|\eta\|_{\ell^\infty(L^2)} \\
	&\leq \|e_h\|_{\ell^\infty(L^2)} + ch^k \| u \|_{\ell^\infty(H^{k+1})}, \\
	\| \nabla (u_h - u) \|_{\ell^2(L^2)} &\leq \| \nabla e_h\|_{\ell^2(L^2)} + \|\nabla \eta\|_{\ell^2(L^2)} \\
	&\leq \|\nabla e_h\|_{\ell^2(L^2)} + ch^k \| u \|_{\ell^2(H^{k+1})}.
\end{align*}
\end{proof}

\begin{remark}
Our analysis need that $Q_h$ is $\pk k$-finite element space in the estimate \eqref{eq:errorEqBAppest} to have $O(h^k)$ in $H^1$-norm.
\end{remark}

\section{Numerical results} \label{sec:numericalResults}
We consider test problems given by manufactured solutions in $d=2$.
We compare Schemes {\schemeone} and {\schemetwo} with $k=2$
for the Oseen problem \eqref{Os}
to show higher accuracy of Scheme {\schemetwo} for small viscosity and large pressures.
We additionally show numerical results of the Navier--Stokes problem, which is given by replacing $w$ by the unknown $u$ in \eqref{Os}.
The corresponding Schemes {\schemethree} and {\schemefour} are given by replacing $w^{n-1}$ by $u_h^{n-1}$ in Schemes {\schemeone} and {\schemetwo}.

\begin{scheme3}
	Let $u_h^0 \in V_h$ be an approximation of $u^0$.
	Find $\{(u_h^n, p_h^n)\}_{n=1}^{N_T} \subset V_h \times \mybar Q_h$ such that
	\begin{alignat*}{2}
	\biggl( \frac{u_h^{n} - u_h^{n-1}\circ X_{1}(\Pi_h^{(1)} u_h^{n-1})}{\Delta t}, v_h \biggr) + a(u_h^n, v_h) + & b(v_h, p_h^n) \qquad \notag \\ & = (f^n,v_h),  ~ && \forall v_h \in V_h, 
	\\
	b(u_h^n, q_h) &= 0, ~ && \forall q_h \in \mybar Q_h.
	\end{alignat*}
\end{scheme3}
\begin{scheme4}
	Let $u_h^0 \in V_h$ be an approximation of $u^0$.
	Find $\{(u_h^n, p_h^n)\}_{n=1}^{N_T} \subset V_h \times Q_h$ such that
	\begin{subequations} \label{eq:schemeNS}
		\begin{align*}
		\biggl( \frac{u_h^{n} - u_h^{n-1}\circ X_{1}(\Pi_h^{(1)} u_h^{n-1})}{\Delta t}, v_h \biggr) + a(u_h^n, v_h) + b(v_h, p_h^n) \qquad \notag \\ = (f^n,v_h),  ~ \forall v_h \in V_h, \\
		b(u_h^n, q_h) - \dzero \ch(p_h^n, q_h) = 0, ~ \forall q_h \in Q_h.
		\end{align*}
	\end{subequations}
\end{scheme4}
In the four schemes
we set the initial value as $u_h^0=\Pi_h^{(2)} u^0$, where $\Pi_h^{(2)}$ is the interpolation operator to the $\pk 2$-element.

\begin{figure}
	\centering
	\includegraphics[width=1.5in]{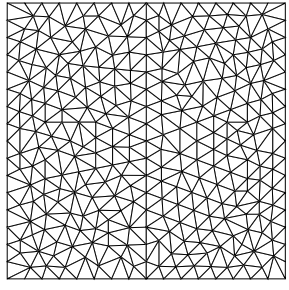}
	\caption{The triangulation of $\mybar \Omega$ for $N=16$ used in Examples 
	\ref{exam:manusol} and \ref{exam:forceNS}.
	}
	\label{fig:irsq_renum}
\end{figure}

\begin{exam} \label{exam:manusol}
We consider the Oseen problem and the Navier--Stokes problem. 
Let $\Omega = (0,1)^2$, $T=1$.
The functions $f$ and $u^0$ are defined so that the exact solution is 
\begin{equation} \label{eq:exactSolForm}
\begin{split}
	u_1(x,t)&=\phi(x_1,x_2,t),\\
	u_2(x,t)&= - \phi(x_2,x_1,t),\\
	p(x,t)&= C_p \sin(\pi(x_1 + 2x_2) + 1 + t),
\end{split}
\end{equation}
where \[\phi(a,b,t) = -\sin(\pi a)^2 \sin(\pi b) \{\sin(\pi(a + t)) + 3\sin(\pi(a + 2b + t))\}.\]
For the Oseen problem we set $w:=u$.
We consider the four cases $\nu=10^{-2}$, $10^{-4}$, $C_p=1$, $10$.
\end{exam}

For triangulations of domains FreeFem++ \cite{FreeFemCite} is used.
Let $N=16, 23, 32 ,45$ and $64$ be the division number of each side of $\mybar \Omega$, and
we set $h = 1/N$. 
Figure \ref{fig:irsq_renum} shows the triangulation of $\mybar \Omega$ when $N=16$. 
The time increment $\Delta t$ is set to be $\Delta t=h^2$ 
so that we can observe the convergence behavior of order $h^2$. 
The purpose of the choice $\Delta t=O(h^2)$ is to examine the theoretical convergence order, but it is not based on the stability condition.
We set the stabilization parameter $\dzero=10^{-1}$ for Schemes {\schemetwo} and {\schemefour}.

The relative error $E_X$ is defined by
\begin{equation*}
E_X(\phi) = \frac{\|\Pi_h^{(m)} \phi-\phi_h \|_X}
{\|\Pi_h^{(m)} \phi\|_X},
\end{equation*}
for $\phi=u$ in $X=\ell^\infty(L^2)$ and $\ell^2(H_0^1)$ with $m=2$, and for $\phi=p$ in $X=\ell^2(L^2)$ with $m=l$ when $\pk 2/ \pk l$-element is used.
Here $\Pi_h^{(m)}$ is the Lagrange interpolation operator to the $\pk m$-finite element space.
Table \ref{table:symbol} shows the symbols used in graphs. 
Since every graph of the relative error $E_X$ versus $h$ is depicted in the logarithmic scale, the slope corresponds to the convergence order. 

\begin{table}
	\caption{Symbols used in Example \ref{exam:manusol}.}
	\centering
	\begin{tabular}{cccc}
		\hline
		$\phi$ 				& $u$ 	& $u$ 	& $p$ \\
		$X$ 				& $\ell^\infty(L^2)$& $\ell^2(H^1_0)$ & $\ell^2(L^2)$ \\
		\hline
		\hline
		TH	 	& \mone & \mtwo & \mthree \\
		Pstab	& \mfour& \mfive& \msix \\
		\hline
	\end{tabular}
	\label{table:symbol}
\end{table}

\paragraph{Case (a)}
Let $C_p=1$ in \eqref{eq:exactSolForm}.
We consider the Oseen problem
and compare Schemes {\schemeone} and \schemetwo.

\begin{figure}
\centering
\includegraphics[height=3in]{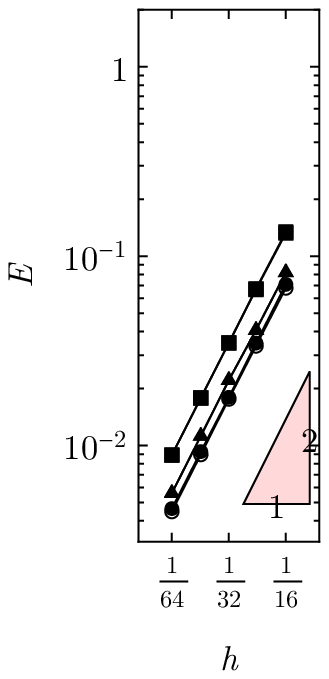} \hspace{10mm}
\includegraphics[height=3in]{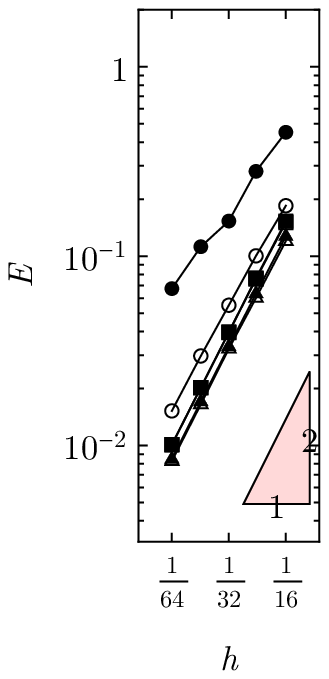}
\caption{Case (a). Relative errors versus $h$ for $\nu=10^{-2}$ (left) and $\nu=10^{-4}$ (right).}
\label{fig:caseA_he}
\end{figure}

Figure \ref{fig:caseA_he} shows the graphs of the errors $E_{\ell^\infty(L^2)}(u)$ (\mone,\mfour), $E_{\ell^2(H^1_0)}(u)$ (\mtwo,\mfive) and $E_{\ell^2(L^2)}(p)$ (\mthree,\msix) versus $h$.
When $\nu=10^{-2}$, all convergence orders are almost two and there are no significant differences in both schemes.

When $\nu=10^{-4}$, the convergence orders of $E_{\ell^\infty(L^2)}(u)$ (\mone,\mfour) are almost two in both schemes and there are no significant differences. 
The values of them are almost 1.5 times larger than those for $\nu=10^{-2}$.
We also get results for $\nu=10^{-6}$, whose graph is omitted here, to observe that increases of the errors compared to $\nu=10^{-4}$ are less than two percent.
The convergence order of $E_{\ell^2(H^1_0)}(u)$ in Scheme {\schemeone} (\mtwo) is less than two, 
while the convergence order is almost two in {\schemetwo} (\mfive) and the value for $N=64$ is four times smaller than that in Scheme {\schemeone}.
In order to obtain the convergence order two in Scheme {\schemeone}, finer meshes seem to be necessary.
The convergence order of the error $E_{\ell^2(L^2)}(p)$ (\mthree,\msix) is almost two in both schemes and the values are almost same as those for $\nu=10^{-4}$.
However, we do not have theoretical estimates for $p$ independent of the viscosity.

We observe that, although in Case (a) there are no significant differences between the both schemes in the errors $E_{\ell^\infty(L^2)}(u)$ (\mone,\mfour), Scheme {\schemetwo} shows higher accuracy for $\nu=10^{-4}$ in the errors $E_{\ell^2(H^1_0)}(u)$ (\mfive).

We consider the problem where the pressure value is larger.
\paragraph{Case (b)}
Let $C_p=10$ in \eqref{eq:exactSolForm}.
We consider the Oseen problem
and compare Schemes {\schemeone} and \schemetwo.

\begin{figure}
\centering
\includegraphics[height=3in]{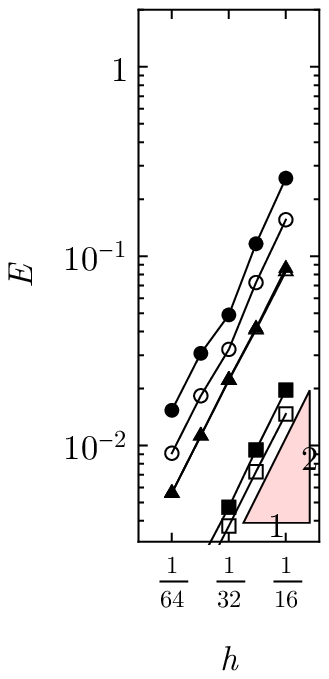} \hspace{10mm}
\includegraphics[height=3in]{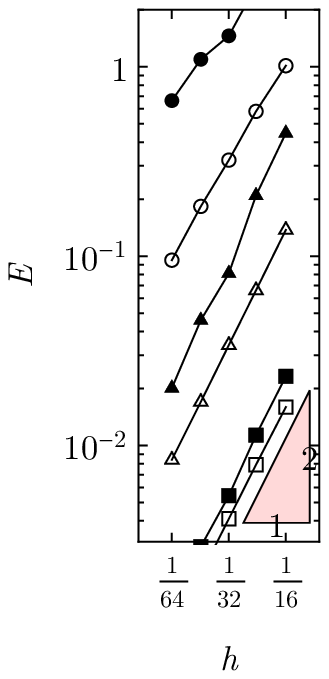}
\caption{Case (b). Relative errors versus $h$ for $\nu=10^{-2}$ (left) and $\nu=10^{-4}$ (right).}
\label{fig:caseOsL10_he}
\end{figure}

Figure \ref{fig:caseOsL10_he} shows the graphs of the errors.
When $\nu=10^{-2}$, the values of $E_{\ell^\infty(L^2)}(u)$ (\mone,\mfour) are almost same as Case (a).
We observe differences in $E_{\ell^2(H^1_0)}(u)$ in the two schemes. 
The values of errors in Scheme {\schemeone} (\mtwo) are about 1.5 times as large as those in Scheme {\schemetwo} (\mfive), and the values in the both schemes are about two to three times as large as in Case (a).
The values of relative errors $E_{\ell^2(L^2)}(p)$ (\mthree,\msix) are, conversely, smaller than those in Case (a).

When $\nu=10^{-4}$, differences of the schemes appear more clearly in $E_{\ell^\infty(L^2)}(u)$ and $E_{\ell^2(H^1_0)}(u)$ than Case (a).
The values of $E_{\ell^\infty(L^2)}(u)$ in Scheme {\schemeone} (\mone) are almost two to three times as large as those in Scheme {\schemetwo} (\mfour).
The values in Scheme {\schemetwo} (\mfour) are almost 1.5 times larger than those for $\nu=10^{-2}$.
We also get results for $\nu=10^{-6}$, whose graph is omitted here, to observe that increases of the errors compared to $\nu=10^{-4}$ are less than 15 percent.
For $N=16$ and $23$ the values of $E_{\ell^2(H^1_0)}(u)$ in Scheme {\schemeone} (\mtwo) are 
too large to be plotted in the graph, and for $N=32,45$ and $64$ the values are
almost four to seven times as large as those in Scheme {\schemetwo} (\mfive).
The values of relative errors $E_{\ell^2(L^2)}(p)$ (\mthree,\msix) are, conversely, smaller than those in Case (a).

We additionally consider the Navier--Stokes problems.

\paragraph{Case (c)}
Let $C_p=1$ in \eqref{eq:exactSolForm}.
We consider the Navier--Stokes problem
and compare Schemes {\schemethree} and \schemefour.

\begin{figure}
\centering
\includegraphics[height=3in]{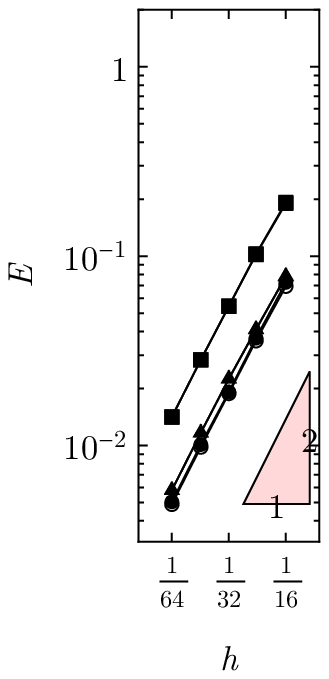} \hspace{10mm}
\includegraphics[height=3in]{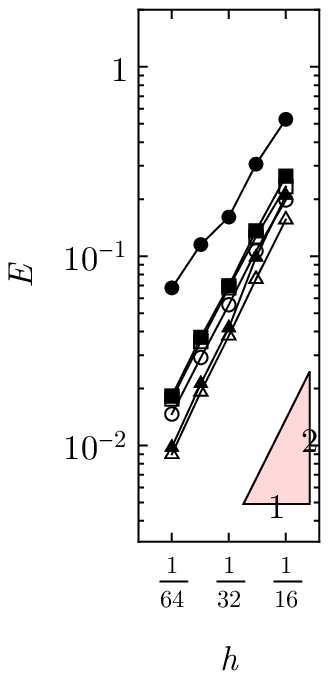}
\caption{Case (c). Relative errors versus $h$ for $\nu=10^{-2}$ (left) and $\nu=10^{-4}$ (right).}
\label{fig:caseNSL1_he}
\end{figure}

Although in \cite{TabataUchiumiNSA} numerical results of Scheme {\schemethree} have already shown, we display them for the sake of completeness.
Figure \ref{fig:caseNSL1_he} shows the graphs of the errors.
We observe the almost same behavior of the errors $E_{\ell^\infty(L^2)}(u)$ (\mone,\mfour) and $E_{\ell^2(H^1_0)}(u)$ (\mtwo,\mfive) as in Case (a) while the values of $E_{\ell^2(L^2)}(p)$ (\mthree,\msix) are almost 1.5 to 2 times as large as in Case (a).

\paragraph{Case (d)}
Let $C_p=10$ in \eqref{eq:exactSolForm}.
We consider the Navier--Stokes problem
and compare Schemes {\schemethree} and \schemefour.

\begin{figure}
\centering
\includegraphics[height=3in]{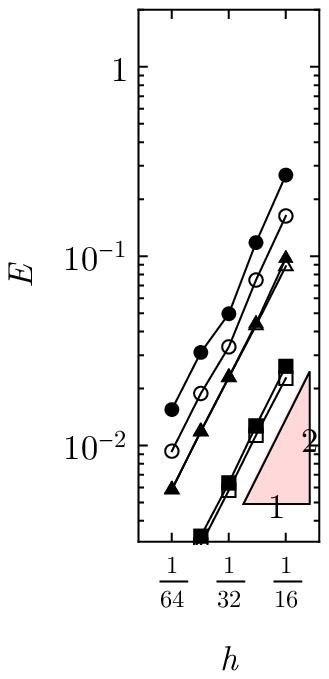} \hspace{10mm}
\includegraphics[height=3in]{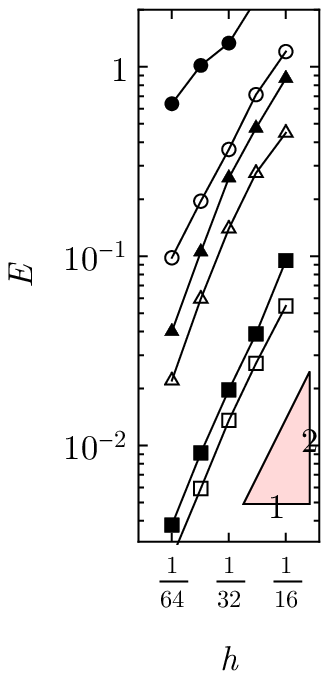}
\caption{Case (d). Relative errors versus $h$ for $\nu=10^{-2}$ (left) and $\nu=10^{-4}$ (right).}
\label{fig:caseNSL10_he}
\end{figure}

Figure \ref{fig:caseNSL10_he} shows the graphs of the errors.
When $\nu=10^{-2}$, we observe the almost same behavior as in Case (b).
When $\nu=10^{-4}$, the values of $E_{\ell^\infty(L^2)}(u)$ (\mone,\mfour) and $E_{\ell^2(L^2)}(p)$ (\mthree,\msix) are almost two to four times as large as in Case (b),
while the values of $E_{\ell^2(H^1_0)}(u)$ (\mtwo,\mfive) are almost same as in Case (b).

\begin{exam} \label{exam:forceNS}
In the Navier--Stokes problem we set
\[
	\Omega = (0, 1)^2, \quad T=40, \quad \nu = 10^{-4}, \quad f(x,t) = (0, 10 \sin(2\pi x_2))^T, \quad u^0 = 0,
\]
and compare Schemes {\schemethree} and \schemefour.
\end{exam}

We can easily check that the solution is $(u, p)(x,t)=(0, -\frac{5}{\pi} \cos(2\pi x_2))$.
We use the mesh shown in Fig.~\ref{fig:irsq_renum} and take $\Delta t = 0.01$.
We set the stabilization parameter $\dzero=10^{-3}$ for Scheme {\schemefour}.

\begin{figure}
	\centering
	\includegraphics[width=2in]{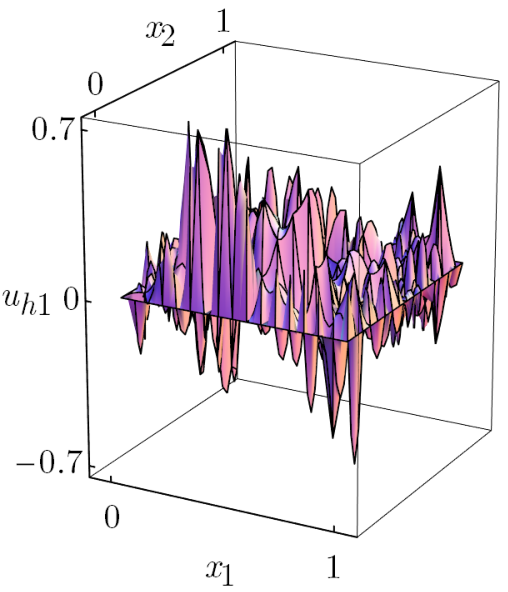} \qquad
	\includegraphics[width=2in]{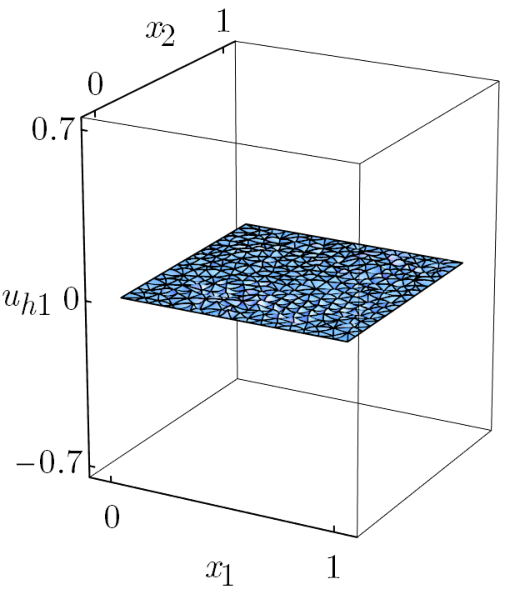} \\
	\includegraphics[width=2in]{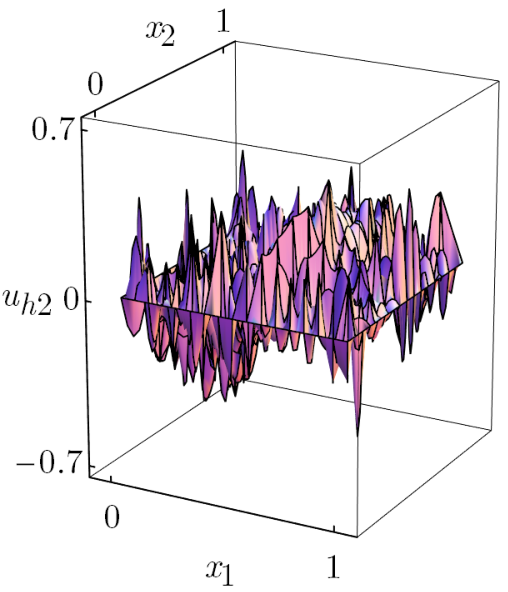} \qquad
	\includegraphics[width=2in]{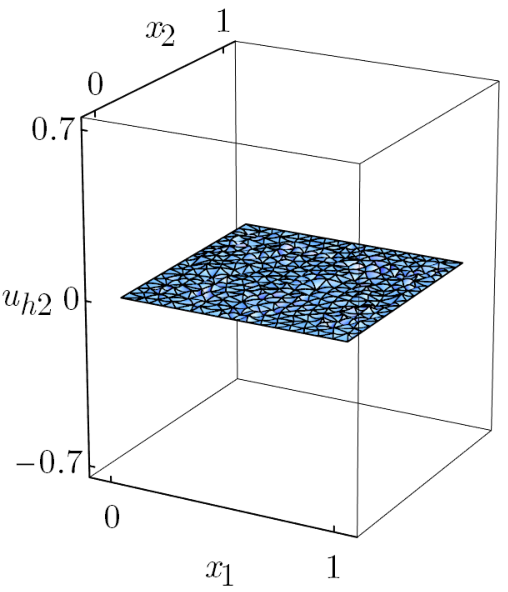} 
	\caption{Example \ref{exam:forceNS}. Stereographs of $u_{h1}^n$ (top) and $u_{h2}^n$ (bottom) at $t^n=40$ by Scheme {\schemethree} (left) and Scheme {\schemefour} (right).}
	\label{fig:ex2u}
\end{figure}
\begin{figure}
	\centering
	\includegraphics[width=1.7in]{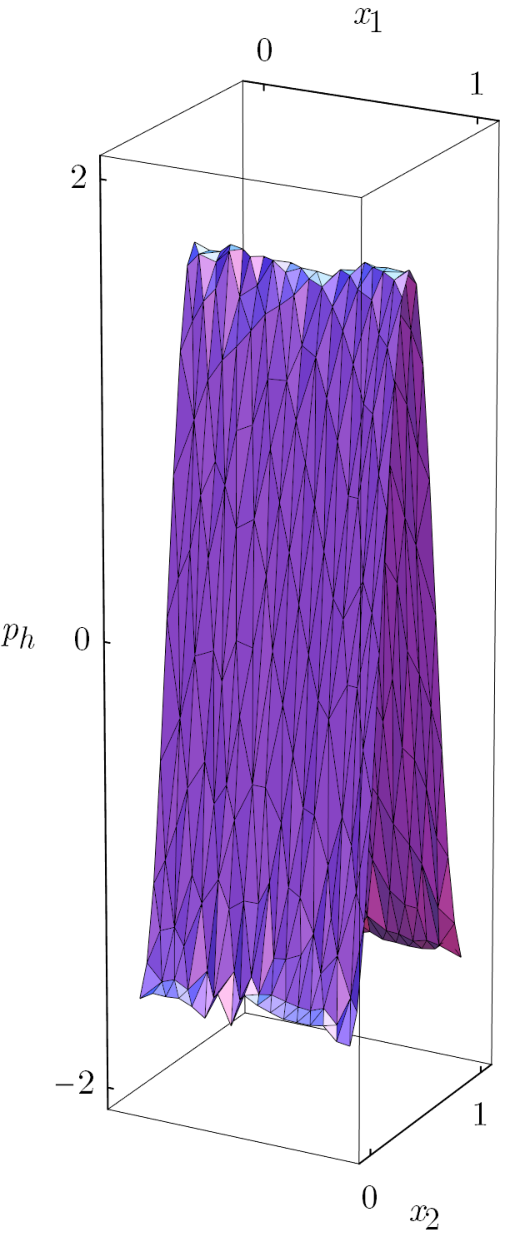} \qquad
	\includegraphics[width=1.7in]{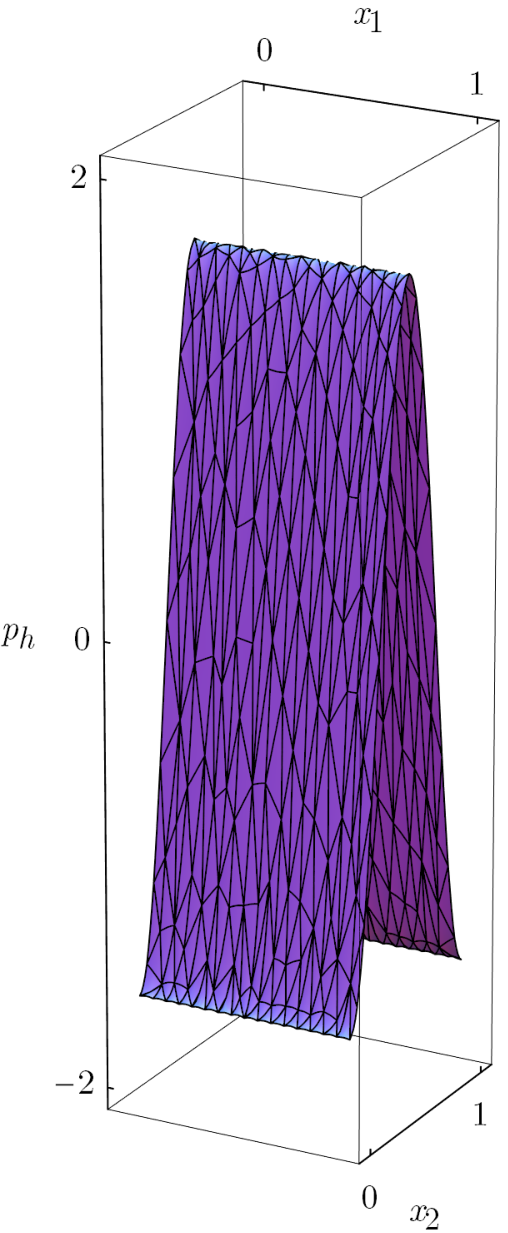}
	\caption{Example \ref{exam:forceNS}. Stereographs of $p_{h}^n$  at $t^n=40$ by Scheme {\schemethree} (left) and Scheme {\schemefour} (right).}
	\label{fig:ex2p}
\end{figure}

Figures \ref{fig:ex2u} and \ref{fig:ex2p} show the stereographs of the solutions $(u_h^n, p_h^n)$ at $t^n=40$ by the both schemes.
In Scheme {\schemethree}, oscillation is clearly observed for both components of the velocity and they are far from the constant zero, while in Scheme {\schemefour} the velocity is almost zero although small ruggedness is observed.
For the pressure, difference between the two schemes is small compared to the velocity but the solution by Scheme {\schemefour} is better.

\section{Concluding remarks} \label{sec:conclusions}
We constructed a pressure-stabilized Lagrange--Galerkin scheme for the Oseen problem with high-order elements, and showed an error estimate with the constant independent of the viscosity.
The numerical examples showed the scheme has higher accuracy than the scheme with Taylor--Hood element especially for problems with small viscosity and large pressures.
(i) Choice of the stabilization parameter in the pressure stabilization term, 
(ii) extension of the discussion to the Navier--Stokes problems, and 
(iii) numerical experiments of physically relevant problems
will be future works.

\paragraph{Acknowledgements}
The author would like to express his gratitude to Professor Emeritus Masahisa Tabata of Kyushu University for valuable discussions and encouragements.
This work was supported 
by Japan Society for the Promotion of Science (JSPS)
under Grant-in-Aid for JSPS Fellows, No. 26$\cdot$964, and under the Japanese-German Graduate Externship (Mathematical Fluid Dynamics), and
by CREST, Japan Science and Technology Agency.

\appendix
\section{Estimates for LG schemes}

Lemma \ref{lemm:jacobiest} is shown in \cite[Lemma 5.7]{TabataUchiumiNSA}.

\begin{lemma} 
	\label{lemm:jacobiest}
	Let $w^*\in W^{1,\infty}(\Omega)^d$ and $X_1(w^*)$ be the mapping defined in \eqref{eq:x1def}.
	Under the condition $\Delta t |w^*|_{1,\infty} \leq 1/4$,
	the estimate 
	\begin{equation*}
	\frac12 \leq \det\biggl(\frac{\partial X_1(w^*)}{\partial x} \biggr) \leq \frac32
	\end{equation*}
	holds,
	where $\det (\partial X_1(w^*)/\partial x)$ is the Jacobian of $X_1(w^*)$.
\end{lemma}

Lemma \ref{lemm:bijective_1} is shown in \cite[Lemma 1]{RuiTabata2002}.
\begin{lemma}
	\label{lemm:bijective_1}
	Let $w^*\in W^{1,\infty}_0(\Omega)^d$ and $X_1(w^*)$ be the mapping defined in \eqref{eq:x1def}.
	Under the condition $\Delta t|w^*|_{1,\infty} \leq 1/4$,
	there exists a positive constant $c$
	independent of $\Delta t$
	such that for $v \in L^2(\Omega)^d$ 
	\begin{equation*}
	\| v \circ X_1(w^*) \|_{0}^{2}
	\leq (1+ c |w^*|_{1,\infty} \Delta t) \|v\|_0^{2}.
	\end{equation*}
\end{lemma}

We now show an estimate for $R_1^n$ in Lemma \ref{lemm:estR1}, or tools for estimating $R_2^n$ and $R_3^n$ in Lemmas \ref{lemm:twofunc} and \ref{lemm:estR3pre}, where $R_i^n$, $i=1,2,3$, are defined in \eqref{eq:defR}.
Although these estimates are frequently used in the analysis of the LG method, e.g. \cite{NotsuTabataOseen2015,TabataUchiumiNSA}, we show proofs of Lemmas \ref{lemm:estR1} and \ref{lemm:estR3pre} for completeness.

\begin{lemma} \label{lemm:estR1}
	Suppose that $u\in Z^2$, $w \in C(W^{1,\infty}_0) \cap  H^1(L^\infty)$ and $\Delta t |w|_{C(W^{1,\infty})} \leq 1/4$. 
	Then
	\begin{align*}
	&\|R_1^n\|_0  \leq \sqrt{\Delta t}  \left[ \sqrt{\frac{2}{3}} (\|w^{n-1}\|_{0,\infty}^2+1)  \|u\|_{Z^2(t^{n-1},t^n)} +  \left\| \frac{\partial w}{\partial t} \right\|_{L^2(t^{n-1}, t^n; L^\infty)} \|\nabla u^n\|_0 \right].
	\end{align*}
\end{lemma}

\begin{proof}
We estimate $\| R_1^n \|_0$ by dividing
\begin{align*}
	R_{1}^n & = 
	\left( \frac{\partial u^n}{\partial t} + (w^{n-1} \cdot \nabla)u^n 
	-\frac{u^n-u^{n-1}\circ X_1(w^{n-1})}{\Delta t} \right)
	+ \left( (w^{n} \cdot \nabla)u^n -  (w^{n-1} \cdot \nabla)u^n \right) \\
	& =: R_{11}^n + R_{12}^n.
\end{align*}
For $R_{11}^n$,
by setting 
\[ 
	y(x,s):= x-sw^{n-1}(x) \Delta t, ~ t(s):=t^n-s\Delta t, 
\]
we use Taylor's theorem to get
\begin{align*}
	&(u^{n-1} \circ X_1(w^{n-1}))(x) = u^n(x) - \Delta t \left( \frac{\partial u^n}{\partial t}  + (w^{n-1} \cdot \nabla) u^n \right)(x) \\
	& \quad + \Delta t^2 \int_0^1 (1-s) \left(  \frac{\partial}{\partial t} + w^{n-1} (x) \cdot \nabla \right)^2 u(y(x,s), t(s)) ds.
\end{align*}
We then have
\begin{align*}
	\|R_{11}^n\|_0 
	&\leq \Delta t \int_0^1 (1-s) \left\|  \left(  \frac{\partial}{\partial t} + w^{n-1} (\cdot) \cdot \nabla \right)^2 u  ( y(\cdot, s), t(s)) \right\|_0 ds\\
	&\leq \sqrt{2/3} \sqrt{\Delta t} (\|w^{n-1}\|_{0,\infty}^2+1)  \|u\|_{Z^2(t^{n-1},t^n)}.
\end{align*}
where we have used the transformation of independent variables from $x$ to $y$ and $s$ to $t$, and the estimate $|\det(\partial x/\partial y)| \leq 2$ by virtue of Lemma \ref{lemm:jacobiest}.
It is easy to show
\[
	\|R_{12}^n\|_0 \leq \sqrt{\Delta t} \left\| \frac{\partial w}{\partial t} \right\|_{L^2(t^{n-1}, t^n; L^\infty)} \|\nabla u^n\|_0.
\]
Combining the two estimate, we have the conclusion. 
\end{proof}

Lemma \ref{lemm:twofunc} is a direct consequence of \cite[Lemma 4.5]{AchdouGuermond2000} and Lemma \ref{lemm:jacobiest}.
\begin{lemma}
	\label{lemm:twofunc}
	Let $1\leq q <\infty$, $1\leq p \leq \infty$, $1/p+1/p'=1$ and
	$w_i\in W_0^{1,\infty}(\Omega)^d$, $i=1,2$.
	Under the condition $\Delta t |w_i|_{1,\infty}\leq 1/4$,
	it holds that,
	for $v \in W^{1,qp'}(\Omega)^d$,
\begin{equation*}\label{eq:twofunc}
	\| v \circ X_1(w_1) - v \circ X_1(w_2)\|_{0,q} \leq 
	2^{1/(qp')}
	\Delta t \|w_1-w_2\|_{0,pq} 
	\|\nabla v \|_{0,qp'},
\end{equation*} 
	where $X_1(\cdot)$ is defined in \eqref{eq:x1def}.
\end{lemma}

\begin{lemma} \label{lemm:estR3pre}
Suppose that $v \in H^1(H^1)$,  $w^*\in W^{1,\infty}_0(\Omega)^d$, and $\Delta t |w^*|_{1,\infty} \leq 1/4$.
Then
\begin{align*}
	&\left\| v^n - v^{n-1} \circ X_1(w^*) \right\|_0 
	\leq \sqrt{2 \Delta t }\biggl(  \left\| \frac{\partial v}{\partial t} \right\|_{L^2(t^{n-1}, t^n ;L^2)} + \|w^*\|_{0,\infty} \|\nabla v\|_{L^2(t^{n-1}, t^n; L^2)} \biggr),
\end{align*}
where $X_1(\cdot)$ is defined in \eqref{eq:x1def}.
\end{lemma}

\begin{proof}
Similarly to the proof of Lemma \ref{lemm:estR1} by defining
\[ 
	y(x,s):= x-sw^{*}(x) \Delta t, ~ t(s):=t^n-s\Delta t, 
\]
we have the estimate
\[
	\left\| v^n - v^{n-1} \circ X_1(w^*) \right\|_0  \leq \Delta t \int_0^1 \left\| \left( \frac{\partial}{\partial t}  +  (w^{*}(\cdot) \cdot \nabla) \right) v (y(\cdot,s), t(s)) \right\|_0 ds.	
\]
The conclusion follows from
the transformation of independent variables from $x$ to $y$ and $s$ to $t$, and the estimate $|\det(\partial x/\partial y)| \leq 2$ by virtue of Lemma \ref{lemm:jacobiest}.
\end{proof}


\end{document}